\documentclass[11pt]{amsart}

\usepackage{amsfonts}
\usepackage{amssymb}
\usepackage{amsthm}
\usepackage{amsmath}
\usepackage{hyperref}
\usepackage{enumitem}
\usepackage{tikz}
\usepackage{pgf}
\usepackage{graphicx}
\usepackage[top=3.5cm, bottom=3.5cm, left=2.5cm, right=2.5cm]{geometry}
\usepackage[capitalise]{cleveref}

\newtheorem{prop}{Proposition}[section]
\newtheorem{theorem}[prop]{Theorem}
\newtheorem{lemma}[prop]{Lemma}
\newtheorem{corollary}[prop]{Corollary}
\newtheorem{conjecture}[prop]{Conjecture}

\theoremstyle{definition}
\newtheorem{definition}[prop]{Definition}

\theoremstyle{remark}

\newtheorem*{remark*}{Remark}
\newtheorem{remark}[prop]{Remark}

\theoremstyle{theorem}

\newcommand{\R}{\mathbb{R}}

\newcommand{\N}{\mathbb{N}}
\newcommand{\Z}{\mathbb{Z}}

\newcommand{\E}{\mathbb{E}}

\newcommand{\F}{\mathbb{F}}

\newcommand{\eps}{\varepsilon}
\newcommand{\pc}{\textup{dc}}
\newcommand{\dc}{\textup{dc}}
\newcommand{\crat}{\textup{cr}}

\newcommand{\supp}{\text{\textup{supp}}\,}

\newcommand*\diff{\mathop{}\!\mathrm{d}}

\numberwithin{equation}{section}

\title{Commuting probabilities of infinite groups}
\date{}
\author{Matthew C. H. Tointon}
\address{Pembroke College, University of Cambridge, CB2 1RF, United Kingdom}
\email{mcht2@cam.ac.uk}
\thanks{This work was supported by grant FN 200021\_163417/1 of the Swiss National Fund for scientific research.}

\begin{document}
\maketitle

\begin{abstract}Let $G$ be a group, and let $M=(\mu_n)_{n=1}^\infty$ be a sequence of finitely supported probability measures on $G$. Consider the probability that two elements chosen independently according to $\mu_n$ commute. Antol\'in, Martino and Ventura define the \emph{degree of commutativity} $\dc_M(G)$ of $G$ with respect to this sequence to be the $\limsup$ of this probability. The main results of the present paper give quantitative algebraic consequences of the degree of commutativity being above certain thresholds. For example, if $\mu_n$ is the distribution of the $n$th step of a symmetric random walk on $G$, or if $G$ is amenable and $(\mu_n)$ is a sequence of almost-invariant measures on $G$, we show that if $\dc_M(G)\ge\alpha>0$ then $G$ contains a normal subgroup $\Gamma$ of index at most $\lceil\alpha^{-1}\rceil$ and a normal subgroup $H$ of cardinality at most $(2\alpha^{-1})^{O(\alpha^{-2}\log\alpha^{-1})}$ such that $H\subset\Gamma$ and such that $\Gamma/H$ is abelian. This generalises a result for finite groups due to P. M. Neumann, and generalises and quantifies a result for certain residually finite groups of subexponential growth due to Antol\'in, Martino and Ventura. We also describe some general conditions on $(\mu_n)$ under which such theorems hold. We close with an application to \emph{conjugacy ratios} as introduced by Cox.
\end{abstract}

\tableofcontents

\section{Introduction}
The starting point for this paper is two simple but beautiful results concerning the probability that two randomly chosen elements of a finite group $G$ commute. If two elements $x,y\in G$ are chosen uniformly at random from $G$ and independently of one another, we define the probability that they commute to be the \emph{degree of commutativity} of $G$, and denote it by $\dc(G)$.
\begin{theorem}[Gustafson {\cite[\S1]{gustafson}}]\label{gustafson}Let $G$ be a finite group such that $\pc(G)>\frac{5}{8}$. Then $G$ is abelian.
\end{theorem}
\begin{theorem}[P. M. Neumann {\cite[Theorem 1]{neumann}}]\label{neumann}
Let $G$ be a finite group such that $\pc(G)\ge\alpha>0$. Then $G$ has a normal subgroup $\Gamma$ of index at most $\alpha^{-1}+1$ and a normal subgroup $H$ of cardinality at most $\exp(O(\alpha^{-O(1)}))$ such that $H\subset\Gamma$ and $\Gamma/H$ is abelian.
\end{theorem}

There are many natural ways in which one might seek to generalise these results. In this paper we are concerned with generalisations to infinite groups. The first question in this setting is how to define the probability that two group elements commute.

Antol\'in, Martino and Ventura \cite{amv} approach this issue by considering sequences of finitely supported probability measures whose supports converge to the whole of $G$. Given a probability measure $\mu$ on $G$, define the \emph{degree of commutativity} $\dc_\mu(G)$ of $G$ with respect to $\mu$ via
\[
\dc_\mu(G)=(\mu\times\mu)(\{(x,y)\in G\times G:xy=yx\}).
\]
Then, given a sequence $M=(\mu_n)_{n=1}^\infty$ of probability measures on $G$, define the \emph{degree of commutativity} $\dc_M(G)$ of $G$ with respect to $M$ via
\[
\dc_M(G)=\limsup_{n\to\infty}\dc_{\mu_n}(G).
\]
In a preliminary version of \cite{amv} available on arxiv.org, Antol\'in, Martino and Ventura suggest the following rather natural analogue of \cref{neumann} for degree of commutativity defined in this way.
\begin{conjecture}[Antol\'in--Martino--Ventura]\label{conj:amv}
For any ``reasonable'' sequence $M=(\mu_n)_{n=1}^\infty$ of probability measures on $G$ we have $\dc_M(G)>0$ if and only if $G$ is virtually abelian.
\end{conjecture}
They suggest that ``reasonable'' might mean that the measures $\mu_n$ cover $G$ with ``enough homogeneity'' as $n\to\infty$. Of course, the terms ``reasonable'' and ``enough homogeneity'' are somewhat vague, and so they also mention two explicit sequences of measures for which they believe \cref{conj:amv} should hold. The first such sequence is where $S$ is a fixed finite symmetric generating set for $G$, and $\mu_n$ is the uniform measure on the ball of radius $n$ about the identity in $G$ with respect to $S$. In this setting we write $\dc_S(G)$ instead of $\dc_M(G)$. The second such sequence is where $\mu$ is some finite probability measure on $G$, and $\mu_n=\mu^{\ast n}$ is defined by letting $\mu^{\ast n}(x)$ be the probability that a random walk of length $n$ on $G$ with respect to $\mu$ ends at $x$.

There has already been some progress towards \cref{conj:amv} in the former setting. Antol\'in, Martino and Ventura themselves prove the conjecture completely in that setting for hyperbolic groups \cite[Theorem 1.9]{amv}, whilst Valiunas proves it for right-angled Artin groups with certain generating sets \cite[Theorem 6]{valiunas}. Antol\'in, Martino and Ventura also prove the following result, which also includes a direct analogue of \cref{gustafson}.
\begin{theorem}[Antol\'in--Martino--Ventura {\cite[Theorem 1.5]{amv}}]\label{amv}
Let $G$ be a finitely generated residually finite group and let $S$ be a finite summetric generating set containing the identity and satisfying
\begin{equation}\label{eq:subexp.growth}
\frac{|S^{n+1}|}{|S^n|}\to1
\end{equation}
as $n\to\infty$. Then $\dc_S(G)>\frac{5}{8}$ if and only if $G$ is abelian, and $\dc_S(G)>0$ if and only if $G$ is virtually abelian.
\end{theorem}
In the present paper we remove completely from \cref{amv} the hypothesis that $G$ is residually finite, and make the second conclusion more quantitative in the spirit of \cref{neumann}.
We also settle \cref{conj:amv} completely for the case in which the sequence of measures is defined via a symmetric random walk.
We do this by offering two possible precise versions of Antol\'in, Martino and Ventura's notion that \cref{conj:amv} should be satisfied by sequences of measures that cover $G$ with ``enough homogeneity'' in the limit, as follows.
\begin{definition}[uniform detection of index]\label{def:detect}
Let $\pi:(0,1]\to(0,1]$ be a non-decreasing function such that $\pi(\gamma)\to0$ as $\gamma\to0$. We say that a sequence $M=(\mu_n)_{n=1}^\infty$ of probability measures on a group $G$ \emph{detects index uniformly at rate $\pi$} if for every $\eps>0$ there exists $N=N(\eps)\in\N$ such that for every $m\in\N$ if $[G:H]\ge m$ then $\mu_n(H)\le\pi(\frac{1}{m})+\eps$ for every $n\ge N$. We also say simply that $M$ \emph{detects index uniformly} to mean that there exists some $\pi$ such that $M$ detects index uniformly at rate $\pi$.
\end{definition}
The word ``uniform'' in \cref{def:detect} refers to the requirement that the definition be satisfied by the same $N(\eps)$ for all subgroups $H$.
\begin{theorem}[weak Neumann-type theorem for measures that detect index uniformly]\label{neumann.detect}
Let $G$ be a group generated by at most $r$ elements, let $M=(\mu_n)_{n=1}^\infty$ be a sequence of measures on $G$ that detects index uniformly at rate $\pi$, and suppose that $\dc_M(G)\ge\alpha>0$. Then $G$ has an abelian subgroup of index at most $O_{r,\pi,\alpha}(1)$.
\end{theorem}
\begin{remark*}
In the converse direction, note that if $G$ has an abelian subgroup $H$ such that $\limsup\mu_n(H)\ge\gamma$ then $\dc_M(G)\ge\gamma^2$.
\end{remark*}
\begin{remark*}
Shalev \cite{shalev} and Martino, Valiunas, Ventura and the author \cite{mtvv} have subsequently extended \cref{neumann.detect} to finitely generated groups in which some higher-weight simple commutator vanishes with positive probability. Such groups have a nilpotent subgroup of bounded index.
\end{remark*}

\begin{definition}[uniform measurement of index]
We say that a sequence $M=(\mu_n)_{n=1}^\infty$ of probability measures on a group $G$ \emph{measures index uniformly} if $\mu_n(xH)\to1/[G:H]$ uniformly over all $x\in G$ and all subgroups $H$ of $G$.
\end{definition}
\begin{remark}\label{measures=>detects}If a sequence of probability measures on a group measures index uniformly then it also detects index uniformly with rate $\iota:(0,1]\to(0,1]$ defined by $\iota(x)=x$.
\end{remark}
\begin{theorem}\label{main:measures}
Let $G$ be a countable group, and let $M=(\mu_n)_{n=1}^\infty$ be a sequence of measures on $G$ that measures index uniformly. Then the following hold.
\begin{enumerate}
\item If $\dc_M(G)>\frac{5}{8}$ then $G$ is abelian.
\item If $\dc_M(G)\ge\frac{1}{2}+\eps$ for some $\eps>0$ then the centre of $G$ has index at most $\frac{1}{\eps}$ in $G$.
\item If $G$ is finitely generated and $\dc_M(G)\ge\alpha>0$ then $G$ has a normal subgroup $\Gamma$ of index at most $\lceil\alpha^{-1}\rceil$ and a normal subgroup $H$ of cardinality at most $(2\alpha^{-1})^{O(\alpha^{-2}\log\alpha^{-1})}$ such that $H\subset\Gamma$ and $\Gamma/H$ is abelian.
\end{enumerate}
\end{theorem}
\begin{remark}\label{rem:fba=>va}
A finitely generated finite-by-abelian group $G$ is virtually abelian, so Theorem \ref{main:measures} implies in particular the conclusion required by \cref{conj:amv}. However, the bound on the index of the abelian subgroup of $G$ obtained in this way depends on the rank of $G$ as well as on $\alpha$, whereas the bound given by Theorem \ref{main:measures} depends only on $\alpha$. The characterisation given by Theorem \ref{main:measures} is therefore quantitatively preferable.

This is slightly reminiscent of Gromov's theorem on groups of polynomial growth. Classically, Gromov's theorem says that a group of polynomial growth is virtually nilpotent. However, Breuillard, Green and Tao \cite[Corollary 11.5]{bgt} have given a quantitative refinement of Gromov's theorem showing that a group $G$ of polynomial growth has a finite-index finite-by-nilpotent subgroup. This implies in particular that $G$ is virtually nilpotent, but this version of the statement offers far less quantitative control. A similar phenomenon occurs in Breuillard and the author's analogous result for finite groups of large diameter \cite[Theorem 4.1]{bt}.
\end{remark}

The following result shows that Theorem \ref{main:measures} applies in particular to a sequence of measures coming from a random walk.
\begin{theorem}[random walks measure index uniformly]\label{rw.uniform}Let $G$ be a finitely generated group, and let $\mu$ be a symmetric, finitely supported generating probability measure on $G$ such that $\mu(e)>0$. Then the sequence $M_\mu=(\mu^{\ast n})_{n=1}^\infty$ measures index uniformly.
\end{theorem}
\begin{remark*}
The well-known example of $G=\Z$, with $\mu$ the uniform probability measure supported on $\pm1$ and $H=2\Z$, shows that \cref{rw.uniform} is not necessarily true in the absence of the condition $\mu(e)>0$ (or some other aperiodicity assumption).
\end{remark*}

Theorem \ref{main:measures} also implies a significant generalisation of Theorem \ref{amv}. First, note that condition \eqref{eq:subexp.growth} is already somewhat suggestive of a particular explicit generalisation, as it implies that the sequence $(S^n)_{n=1}^\infty$ is a \emph{F\o lner sequence} for the group $G$. In general a sequence $(F_n)_{n=1}^\infty$ of finite subsets of $G$ is said to be a \emph{left-F\o lner sequence}, or simply a \emph{F\o lner sequence}, if
\[
\frac{|xF_n\triangle F_n|}{|F_n|}\to0
\]
for every $x\in G$. Even more generally, for a given F\o lner sequence $(F_n)_{n=1}^\infty$, if $\mu_n$ is the uniform probability measure on the set $F_n$ then the sequence $(\mu_n)_{n=1}^\infty$ satisfies the condition
\begin{equation}\label{eq:almost.inv}
\|x\cdot\mu_n-\mu_n\|_1\to0
\end{equation}
for every $x\in G$ (here $x\cdot\mu$ is defined by setting $x\cdot\mu(A)=\mu(x^{-1}A)$). A sequence $(\mu_n)_{n=1}^\infty$ of probability measures satisfying \eqref{eq:almost.inv} is said to be a sequence of \emph{almost-invariant probability measures}. The following result therefore combines with Theorem \ref{main:measures} and Remark \ref{rem:fba=>va} to recover in particular Theorem \ref{amv} without the assumption of residual finiteness.

\begin{theorem}[almost-invariant measures measure index uniformly]\label{folner.uniform}Let $G$ be a finitely generated group. Then every sequence of almost-invariant measures on $G$ measures index uniformly.
\end{theorem}
\begin{remark*}
By contrast with \cref{folner.uniform}, if $\mu_n$ is the uniform measure on the ball of radius $n$ in a group of exponential growth then the sequence $(\mu_n)_{n=1}^\infty$ may not even \emph{detect} index uniformly. For example, writing $F(x,y)$ for the free group on generators $x,y$, and taking $G=\Z\times F(x,y)$ with generating set $\{(0,e),(\pm1,e),(0,x^{\pm1}),(0,y^{\pm1})\}$, the subgroup $F(x,y)$ has infinite index but $\mu_n(F(x,y))\not\to0$.
\end{remark*}

It is well known and easy to check that if $G$ is a group of subexponential growth and $S$ is an arbitrary finite symmetric generating set for $G$ containing the identity then there is a sequence $n_1<n_2<\ldots$ such that $(S^{n_i})_{i=1}^\infty$ is a F\o lner sequence for $G$. Theorems \ref{main:measures} and \ref{folner.uniform} therefore also have the following immediate corollary.
\begin{corollary}
Let $G$ be a finitely generated group of subexponential growth with finite symmetric generating set $S$ containing the identity, and for each $n$ let $\mu_n$ be the uniform probability measure on $S^n$. Then the following hold.
\begin{enumerate}
\item If $\liminf\dc_{\mu_n}(G)>\frac{5}{8}$ then $G$ is abelian.
\item If $\liminf\dc_{\mu_n}(G)\ge\frac{1}{2}+\eps$ for some $\eps>0$ then the centre of $G$ has index at most $\frac{1}{\eps}$ in $G$.
\item If $\liminf\dc_{\mu_n}(G)\ge\alpha>0$ then $G$ has a normal subgroup $\Gamma$ of index at most $\lceil\alpha^{-1}\rceil$ and a normal subgroup $H$ of cardinality at most $(2\alpha^{-1})^{O(\alpha^{-2}\log\alpha^{-1})}$ such that $H\subset\Gamma$ and $\Gamma/H$ is abelian.
\end{enumerate}
\end{corollary}
Of course, for \cref{conj:amv} to hold in this case we would need to replace the $\liminf$s with $\limsup$s.

Let us emphasise that there are many groups that do not admit generating sets satisfying condition \eqref{eq:subexp.growth}, but that nonetheless admit sequences of almost-invariant measures. Indeed, it is well known that a countable discrete group admits such a sequence if and only if it belongs to the class of \emph{amenable} groups, which includes all soluble groups, for example, whereas a generating set can satisfy \eqref{eq:subexp.growth} only in a group of subexponential growth.

Unfortunately we do not resolve \cref{conj:amv} for the seqence of uniform probability measures on balls with respect to a fixed generating set that does not satisfy \eqref{eq:subexp.growth}. As is noted in \cite{amv}, in order to prove the conjecture in this setting it would be sufficient to prove that if a finite generating set $S$ for a group $G$ did not satisfy \eqref{eq:subexp.growth} then $\dc_S(G)=0$.

\bigskip
\noindent\textsc{Degree of commutativity with respect to an invariant mean.} We now present a variant of \cref{main:measures} for amenable groups, which is both a natural result in its own right and, as it turns out, an ingredient in the proof of \cref{main:measures} (see \cref{rem:mean} for more details).

By definition, a group is amenable if it admits a \emph{finitely additive left-invariant mean}. A \emph{finitely additive mean} on a group $G$ is a linear functional $\int\diff\mu:\ell^\infty(G)\to\R$ that is positive in the sense that $\int f\diff\mu\ge0$ if $f\ge0$ pointwise, and normalised in the sense that $\int 1_G\diff\mu=1$. Such a mean is said to be \emph{left-invariant} if, defining $g\cdot f$ via $g\cdot f(x)=f(g^{-1}x)$ for $g\in G$, we have $\int g\cdot f\diff\mu=\int f\diff\mu$ for every $f\in\ell^\infty(G)$ and every $g\in G$. In general we abuse notation slightly and denote the mean $\int\diff\mu$ simply by $\mu$. Given a subset $X\subset G$ and a function $f\in\ell^\infty(G)$, we write $1_X$ for the characteristic function of $X$, write
\[
\int_{x\in X}f(x)\diff\mu(x)=\int1_Xf\diff\mu,
\]
and abbreviate
\[
\mu(X)=\int1_X\diff\mu.
\]

It is very natural to define the degree of commutativity on an amenable group $G$ via a finitely additive left-invariant mean $\mu$. Given such a mean on $G$, we define the \emph{product mean} $\mu\times\mu$ on $G\times G$ via
\[
\int_{G\times G}f\diff(\mu\times\mu)=\int_{x\in G}\int_{y\in G}f(x,y)\diff\mu(x)\diff\mu(y).
\]
We then define
\[
\pc_\mu(G)=(\mu\times\mu)(\{(x,y)\in G\times G:xy=yx\}).
\]

\begin{theorem}\label{main:mean}
Let $G$ be an amenable group with finitely additve left-invariant mean $\mu$. Then the following hold.
\begin{enumerate}
\item If $\dc_\mu(G)>\frac{5}{8}$ then $G$ is abelian.
\item If $\dc_\mu(G)\ge\frac{1}{2}+\eps$ for some $\eps>0$ then the centre of $G$ has index at most $\frac{1}{\eps}$ in $G$.
\item If $\dc_\mu(G)\ge\alpha>0$ then $G$ has a normal subgroup $\Gamma$ of index at most $\lceil\alpha^{-1}\rceil$ and a normal subgroup $H$ of cardinality at most $(2\alpha^{-1})^{O(\alpha^{-2}\log\alpha^{-1})}$ such that $H\subset\Gamma$ and $\Gamma/H$ is abelian.
\end{enumerate}
\end{theorem}
Let us emphasise that we do not assume in \cref{main:mean} that $G$ is finitely generated.

\begin{prop}[converse to \cref{main:measures} (3) and \cref{main:mean} (3)]\label{neum.converse}
Let $G$ be a countable group with a subgroup $\Gamma$ of index $m$ and a subgroup $H\lhd\Gamma$ of cardinality $d$ such that $\Gamma/H$ is abelian. Then for every finitely additive left-invariant mean $\mu$ on $G$ we have $\pc_\mu(G)\ge\frac{1}{m^2d}$. Moreover, if $G$ is finitely generated and $M$ is a sequence of probability measures on $G$ that measures index uniformly then $\pc_M(G)\ge\frac{1}{m^2d}$.
\end{prop}
\begin{remark*}
For non-finitely generated groups it is not the case that finite-by-abelian implies virtually abelian. An example showing this appears on Tao's blog \cite{tao.blog}: if $V$ is an infinite vector space over a finite field $\F$ and $b:V\times V\to\F$ is a non-degenerate anti-symmetric bilinear form then the product set $V\times\F$ with group operation defined by $(v,x)(w,y)=(v+w,x+y+b(v,w))$ for $v,w \in V$ and $x,y\in F$ is finite-by-abelian but not virtually abelian. \cref{neum.converse} therefore shows that the characterisation proposed by \cref{conj:amv} does not hold for groups that are not finitely generated when degree of commutativity is defined with respect to an invariant mean. 
\end{remark*}

\bigskip

\noindent\textsc{Independence of degree of commutativity from the method of averaging.} In a forthcoming paper, Antol\'in, Martino and Ventura show that in a virtually abelian group $G$ the value of $\dc_S(G)$ is independent of the choice of generating set $S$, and that the $\limsup$ appearing in the definition of $\dc_S(G)$ is actually a limit. This leads immediately to the following corollary of \cref{amv}.
\begin{corollary}[Antol\'in--Martino--Ventura, unpublished]\label{amv.indep}
Let $G$ be a finitely generated residually finite group and let $S,S'$ be two generating sets, each satisfying \eqref{eq:subexp.growth}. Then $\dc_S(G)=\dc_{S'}(G)$, and the $\limsup$ appearing in the definition of $\dc_S(G)$ is actually a limit.
\end{corollary}

Here we prove a similar result, which appears to be a necessary ingredient in our proof of \cref{main:measures} (see \cref{rem:mean}).

\begin{theorem}\label{indep}
Let $G$ be a finitely generated group, and let $M$ be a sequence of probability measures on $G$ that measures index uniformly. Then the $\limsup$ in the definition of $\dc_M(G)$ is actually a limit, and $\dc_{M'}(G)=\dc_M(G)$ for every other sequence $M'$ of probability measures on $G$ that measures index uniformly. Moreover, if $G$ is amenable then $\dc_\mu(G)=\dc_M(G)$ for every finitely additive left-invariant mean $\mu$ on $G$.
\end{theorem}

As a by-product of the proofs of our theorems we also obtain the following partial version of \cref{indep} for groups that are not finitely generated.
\begin{corollary}\label{indep.from.mu.and.F}
Let $G$ be a countable group. Then the following hold.
\begin{enumerate}
\item If $\mu$ is a finitely additive left-invariant mean on $G$ such that $\pc_\mu(G)=0$ then $\pc_{\mu'}(G)=0$ for every other finitely additive left-invariant mean $\mu'$ on $G$.
\item If $\alpha>\frac{1}{2}$, and if there is a finitely additive left-invariant mean $\mu$ on $G$ such that $\pc_\mu(G)=\alpha$ or some sequence $M$ of probability measures on $G$ that measures index uniformly such that $\pc_M(G)=\alpha$, then $\pc_{\mu'}(G)=\alpha$ for every finitely additive left-invariant mean $\mu'$ on $G$, and $\pc_{M'}(G)=\alpha$ for every sequence $M'$ of probability measures on $G$ that measures index uniformly. Moreover, the $\limsup$ in the definition of $\dc_{M'}(G)$ is actually a limit.
\end{enumerate}
\end{corollary}

\begin{remark*}
It follows from \cref{main:mean} and \cref{neum.converse} that in an arbitrary amenable group $G$ the degree of commutativity $\pc_\mu(G)$ with respect to a given finitely additive left-invariant mean $\mu$ is bounded from below and above in terms of $\pc_{\mu'}(G)$ for any other finitely additive left-invariant mean $\mu'$. It would be interesting to see whether different invariant means really can give different values of $\pc$ in situations not governed by \cref{indep} and \cref{indep.from.mu.and.F}, which is to say when they take values in $(0,\frac{1}{2}]$ in a group that is not finitely generated.
\end{remark*}

\begin{remark*}Martino, Valiunas, Ventura and the author \cite{mtvv} have subsequently generalised \cref{indep} to cover the probability that an arbitrary equation in the elements of $G$ is satisfied.
\end{remark*}

\bigskip
\noindent\textsc{Outline of the paper.} We prove \cref{neumann.detect} in \cref{sec:weak}, \cref{rw.uniform} in \cref{sec:uniform.rw}, and \cref{folner.uniform} in \cref{sec:uniform}. We prove \cref{indep} in \cref{sec:indep}, and Theorems \ref{main:measures} and \ref{main:mean}, \cref{neum.converse}, and \cref{indep.from.mu.and.F} in \cref{sec:proofs}. In \cref{sec:cr} we present an application of our results to \emph{conjugacy ratios} as introduced by Cox \cite{cox}.

\bigskip
\noindent\textsc{Acknowledgements}
This project was inspired by an interesting talk on the work \cite{amv} given by Laura Ciobanu at the Universit\'e de Neuch\^atel in May 2017. I am grateful to her both for this talk and for a very useful subsequent discussion. I am also grateful to Itai Benjamini, Emmanuel Breuillard, Ben Green and Alain Valette for helpful conversations, and to Yago Antol\'in, Dan Segal, Enric Ventura and an anonymous referee for helpful comments on earlier versions of this manuscript.

\section{A weak Neumann-type theorem}\label{sec:weak}
In this section we prove \cref{neumann.detect}. We start with the following result, which is based on a simple technique used by Neumann \cite{neumann} that we use repeatedly throughout this paper.
\begin{prop}\label{large.centraliser.detect}
Let $M=(\mu_n)_{n=1}^\infty$ be a sequence of measures that detects index uniformly at rate $\pi$. Let $\alpha\in(0,1]$, and suppose that $\dc_M(G)\ge\alpha$. Let $\gamma\in(0,1)$ be such that $\pi(\gamma)<\alpha$, and write
\[
X=\{x\in G:[G:C_G(x)]\le\textstyle\frac{1}{\gamma}\}.
\]
Then $\limsup_{n\to\infty}\mu_n(X)\ge\alpha-\pi(\gamma)$.
\end{prop}
\begin{proof}
By definition of $\dc_M$ there exists a sequence $n_1<n_2<\ldots$ such that $\dc_{\mu_{n_i}}(G)\ge\alpha-o(1)$. Writing $\E^{(n)}$ for expectation with respect to $\mu_n$, this means precisely that
\[
\E^{(n_i)}_{x\in G}(\mu_{n_i}(C_G(x)))\ge\alpha-o(1).
\]
Following Neumann \cite{neumann}, we note that therefore
\begin{align*}
\alpha&\le\mu_{n_i}(X)\E^{(n_i)}_{x\in X}(\mu_{n_i}(C_G(x)))+\mu_{n_i}(G\backslash X)\E^{(n_i)}_{x\in G\backslash X}(\mu_{n_i}(C_G(x)))+o(1)\\
    &\le\mu_{n_i}(X)+\pi(\gamma)+o(1),
\end{align*}
the last inequality being by uniform detection of index, and the result follows.
\end{proof}

\begin{lemma}[{\cite[Proposition 1.1.1]{ls}}]\label{fin.gen.fin.ind}
Let $m,r\in\N$, and let $G$ be a group generated by $r$ elements. Then $G$ has at most $O_{m,r}(1)$ subgroups of index $m$.
\end{lemma}

\begin{proof}[Proof of \cref{neumann.detect}]
Let $\gamma=\frac{1}{2}\inf\{\beta\in(0,1]:\pi(\beta)\ge\frac{\alpha}{2}\}$, noting that therefore $\pi(\gamma)<\frac{\alpha}{2}$, and write $X=\{x\in G:[G:C_G(x)]\le\frac{1}{\gamma}\}$, so that \cref{large.centraliser.detect} gives $\limsup_{n\to\infty}\mu_n(X)>\frac{\alpha}{2}$. This implies in particular that the group $\Gamma$ generated by $X$ satisfies $\limsup_{n\to\infty}\mu_n(\Gamma)>\frac{\alpha}{2}$, and hence $[G:\Gamma]\le\frac{1}{\gamma}$ by the definitions of $\gamma$ and uniform detection of index at rate $\pi$.

 \cref{fin.gen.fin.ind} and the definition of $\gamma$ imply that there are at most $O_{r,\pi,\alpha}(1)$ possibilities for $C_G(x)$ with $x\in X$, and so their intersection has index at most $O_{r,\pi,\alpha}(1)$ in $G$. However, this intersection is contained in $C_G(\Gamma)$, and so its intersection with $\Gamma$, which again has index at most $O_{r,\pi,\alpha}(1)$ in $G$, is abelian.
\end{proof}

\section{Uniform measurement of index by random walks}\label{sec:uniform.rw}
In this section we prove \cref{rw.uniform}. Our key tools will be results of Morris and Peres \cite{mp} concerning mixing time and heat-kernel decay of Markov chains. Following the set-up of their paper, let $\{p(x,y)\}$ be the transition probabilities for an irreducible Markov chain on countable state space $V$. A measure $\pi$ on $V$ is said to be \emph{stationary} for $p$ if $\sum_{x\in V}\pi(x)p(x,y)=\pi(y)$ for all $x\in V$; if $V$ is finite we also require $\pi$ to be a probability measure.
\begin{lemma}\label{unif.stationary}
If $p$ is symmetric and $\pi$ is uniform then $\pi$ is stationary for $p$.
\end{lemma}
\begin{proof}
We have $\sum_{x\in V}p(x,y)=\sum_{x\in V}p(y,x)=1$.
\end{proof}
Given a stationary measure $\pi$ on $V$, write $\pi_\ast=\inf_{x\in V}\pi(x)$. For $A\subset V$ write
\[
|\partial A|=\sum_{x\in A,y\notin A}\pi(x)p(x,y),
\]
and define the \emph{conductance} $\Phi_A$ of $A\ne\varnothing$ via
\[
\Phi_A=\frac{|\partial A|}{\pi(A)}.
\]
Define $\Phi:[\pi_\ast,\infty)\to\R$ via
\[
\Phi(r)=\left\{\begin{array}{cl}
                            \inf\{\Phi_A:\pi(A)\le r\} &  \text{if $r\le\frac{1}{2}$}\\
                            \Phi(\frac{1}{2}) & \text{otherwise}
                      \end{array}\right.
\]
if $V$ is finite, or via
\[
\Phi(r)=\inf\{\Phi_A:\pi(A)\le r\}
\]
if $V$ is infinite. Note that in each case irreducibility implies that
\begin{equation}\label{eq:Phi.min}
\Phi(r)\ge\frac{\inf\{\pi(x)p(x,y):p(x,y)>0\}}{r}
\end{equation}
for every $r$.
\begin{theorem}[Morris--Peres \cite{mp}]\label{thm:mp}
Let $\eps>0$, and let $c\in(0,\frac{1}{2}]$ be such that $p(x,x)\ge c$ for all $x\in V$. Then for every
\begin{equation}\label{eq:thm.n.large}
n\ge1+\frac{(1-c)^2}{c^2}\int_{4\pi_\ast}^{4/\eps}\frac{4\diff u}{u\Phi(u)^2}
\end{equation}
we have
\[
\left|\frac{p^n(x,y)-\pi(y)}{\pi(y)}\right|\le\eps
\]
if $V$ is finite \cite[Theorem 5]{mp}, or
\[
\left|\frac{p^n(x,y)}{\pi(y)}\right|\le\eps
\]
if $V$ is infinite \cite[Theorem 2]{mp}.
\end{theorem}
\begin{proof}[Proof of \cref{rw.uniform}]
Write $c=\min_{x\in\supp\mu}\mu(x)$. Let $\eps>0$ and let $H$ be a subgroup of $G$. We will prove that for every
\begin{equation}\label{eq:n.size}
n\ge1+\frac{32(1-c)^2}{c^4\eps^2}
\end{equation}
we have
\begin{equation}\label{eq:target.finite}
\left|\mu^{\ast n}(xH)-\textstyle\frac{1}{[G:H]}\right|\le\eps
\end{equation}
for every $x\in G$ if $H$ has finite index, or
\begin{equation}\label{eq:target.infinite}
|\mu^{\ast n}(xH)|\le\eps
\end{equation}
for every $x\in G$ if $H$ has infinite index; in particular, this implies the theorem.

The left random walk on $G$ with respect to $\mu$ induces an irreducible Markov chain on the left cosets $xH$ of $H$. Let $p$ be the transition matrix of this Markov chain. The symmetry of $\mu$ implies that $p$ is symmetric, and hence, by \cref{unif.stationary}, that we may assume that $\pi$ is the uniform probability measure on $G/H$ if $H$ has finite index, or that $\pi$ is everywhere equal to $1$ if $H$ has infinite index. The inequality \eqref{eq:target.finite} is therefore equivalent to the inequality
\[
\left|\frac{p^n(H,xH)-\pi(xH)}{\pi(xH)}\right|\le\eps[G:H],
\]
whilst the inequality \eqref{eq:target.infinite} is equivalent to the inequality
\[
\left|\frac{p^n(H,xH)}{\pi(xH)}\right|\le\eps.
\]
However, since $\inf\{p(xH,yH):p(xH,yH)>0\}\ge c$, it follows from \eqref{eq:Phi.min} that the integrand of \eqref{eq:thm.n.large} is at most $4u[G:H]^2/c^2$ if $H$ has finite index, or at most $4u/c^2$ if $H$ has infinite index, and so \cref{thm:mp} implies that \eqref{eq:target.finite} is satisfied if
\[
n\ge1+\frac{4(1-c)^2[G:H]^2}{c^4}\int_{0}^{4/(\eps[G:H])}u\diff u,
\]
or that \eqref{eq:target.infinite} is satisfied if
\[
n\ge1+\frac{4(1-c)^2}{c^4}\int_{0}^{4/\eps}u\diff u.
\]
In each case this is equivalent to \eqref{eq:n.size}, as required.
\end{proof}

\section{Uniform measurement of index by almost-invariant measures}\label{sec:uniform}
In this section we prove \cref{folner.uniform}. Throughout the section, in order to avoid having to repeat arguments that are essentially identical for finite- and infinite-index subgroups, we adopt the convention that if $M$ is an infinite quantity (such as the index of a subgroup or the number of vertices in a graph) then $1/M$ takes the value zero.

The only tool we really need other than the almost-invariance of the measures under consideration is the following standard fact.
\begin{lemma}\label{reps.in.bdd.ball}
Let $G$ be a group with finite symmetric generating set $S$ containing the identity, let $H$ be a subgroup of index at least $m\in\N$, and let $x\in G$. Then there exist $y_1,\ldots,y_m\in S^m$ such that the elements $y_ix$ belong to distinct left cosets of $H$.
\end{lemma}
\begin{proof}
We essentially reproduce the proof of the closely related \cite[Lemma 2.7]{bt}. If $S^{n+1}xH=S^nxH$ for a given $n$ then it follows by induction that $S^rxH=S^nxH$ for every $r\ge n$, and hence that $G=S^nxH$. We may therefore assume that $SxH\subsetneq S^2xH\subsetneq\ldots\subsetneq S^mxH$. However, this implies that the number of left cosets of $H$ having non-empty intersection with $S^nx$ is strictly increasing for $n=1,\ldots,m$, and so the lemma is proved. 
\end{proof}

\begin{proof}[Proof of \cref{folner.uniform}]
It follows from the definition of a sequence of almost-invariant measures that for every $y\in G$ we have
\begin{equation}\label{eq:almost-inv.unif}
|\mu_n(yxH)-\mu_n(xH)|\to0
\end{equation}
uniformly over all $x\in G$ and all subgroups $H$ of $G$. Since this convergence is uniform over all $y\in S^m$, \cref{reps.in.bdd.ball} implies that
\[
\mu_n(xH)\to\frac{1}{[G:H]}
\]
uniformly over all $x\in G$ and all subgroups $H$ of $G$ of index at most $m$. On the other hand, given $\eps>0$, let $m$ be the smallest integer greater than $\frac{1}{\eps}$. If $H$ has index greater than $m$ then \cref{reps.in.bdd.ball} and the fact that the convergence in \eqref{eq:almost-inv.unif} is uniform over $y\in S^{m+1}$ implies that provided $n$ is large enough in terms of $\eps$ we have $\mu_n(xH)<\eps$, and hence in particular
\[
\left|\mu_n(xH)-\frac{1}{[G:H]}\right|\le\eps.
\]
The convergence is therefore uniform over all $x\in G$ and all subgroups $H$ of $G$, as required.
\end{proof}

\section{A finite-index subgroup with uniformly bounded conjugacy classes}\label{sec:bfc}
At the heart of Neumann's original proof of \cref{neumann} lies a result essentially saying that if many elements of $G$ have bounded conjugacy classes then $G$ has a large subgroup in which \emph{every} element has a bounded conjugacy class. In this section we generalise his argument to arbitrary amenable groups as follows.
\begin{prop}\label{indep.cosets}
Let $G$ be an amenable group with finitely additive left-invariant mean $\mu$, let $\alpha\in(0,1]$, and suppose that $\pc_\mu(G)\ge\alpha$. Given $r\in\N$, write $Y_r=\{x\in G:[G:C_G(x)]\le r\}$, and write $\Gamma_r$ for the group generated by $Y_r$. Then
\begin{enumerate}
\renewcommand{\labelenumi}{(\Alph{enumi})}
\item $[G:\Gamma_r]\le\frac{1}{\alpha-1/r}$, and
\item if $r\ge\frac{2}{\alpha}$ then $[G:C_G(x)]\le r^{6/\alpha+2}$ for every $x\in\Gamma_r$.
\end{enumerate}
\end{prop}
The following result then shows that this large subgroup has a small commutator subgroup.
\begin{theorem}[Guralnick--Mar\'oti {\cite[Theorem 1.9]{gm}}]\label{bfc}
Let $G$ be a group in which the size of every conjugacy class is at most $k>1$. Then $|[G,G]|\le k^{\frac{1}{2}(7+\log_2k)}$.
\end{theorem}

The first step of Neumann's original proof of \cref{neumann} is essentially the following observation.
\begin{lemma}\label{large.centraliser}
Let $G$ be an amenable group with finitely additive left-invariant mean $\mu$, let $\alpha\in(0,1]$, and suppose that $\pc_\mu(G)\ge\alpha$. Let $r>\frac{1}{\alpha}$ and write $X=\{x\in G:[G:C_G(x)]\le r\}$. Then $\mu(X)\ge\alpha-\frac{1}{r}$.
\end{lemma}
\begin{proof}
We have
\begin{align*}
\pc_\mu(G)&=\int_{x\in X}\mu(C_G(x))\diff\mu(x)+\int_{x\in G\backslash X}\mu(C_G(x))\diff\mu(x)\\
    &\le\mu(X)+\textstyle\frac{1}{r}.
\end{align*}
\end{proof}

\begin{lemma}\label{neum.translates}
Let $G$ be an amenable group with finitely additive left-invariant mean $\mu$, let $X\subset G$ be a symmetric subset containing the identity, and let $m\in\N$. Suppose that
\begin{equation}\label{eq:neum.trans}
\mu(X)\ge\frac{1}{m}.
\end{equation}
Then $\langle X\rangle=X^{3m-1}$.
\end{lemma}
\begin{proof}
We follow Eberhard \cite[Lemma 2.1]{sean}, whose proof for the finite case streamlines part of Neumann's argument \cite{neumann}. Select for as long as possible a sequence $x_1,x_2,\ldots$ in $G$ such that $x_i\in X^{3i}\backslash X^{3i-1}$. If we have such a sequence $x_1,\ldots,x_j$ then $x_iX\subset X^{3i+1}\backslash X^{3i-2}$, and so the sets $X$ and $x_iX$ are all disjoint. The inequality \eqref{eq:neum.trans} and the left-invariance of $\mu$ therefore imply that $j\le m-1$. It follows that $X^{3m}=X^{3m-1}$, and hence that $\langle X\rangle=X^{3m-1}$, as required.
\end{proof}

\begin{proof}[Proof of \cref{indep.cosets}]
We essentially follow Neumann \cite[Theorem 1]{neumann}. \cref{large.centraliser} implies that $\mu(Y_r)\ge\alpha-\textstyle\frac{1}{r}$, which implies (A). Moreover, if $r\ge\frac{2}{\alpha}$ then by \cref{neum.translates} it also implies that
\begin{equation}\label{eq:Gamma.in.Y}
\Gamma_r=Y_r^{3\lceil\frac{2}{\alpha}\rceil-1}.
\end{equation}
Writing $k_G(x)$ for the size of the conjugacy class in $G$ of $x\in G$, and noting that $k_G(x)=[G:C_G(x)]$ and $k_G(xy)\le k_G(x)k_G(y)$ for every $x,y\in G$, conclusion (B) then follows from \eqref{eq:Gamma.in.Y} and the definition of $Y_r$.
\end{proof}

\section{Independence of degree of commutativity from the method of averaging}\label{sec:indep}
\begin{proof}[Proof of \cref{indep}]
If $\dc_M(G)>0$ for some sequence $M$ of probability measures that detects index uniformly then \cref{neumann.detect} and \cref{measures=>detects} imply that $G$ is virtually abelian, and in particular amenable, and it is then trivial that for every finitely additive left-invariant mean $\mu$ we have $\dc_\mu(G)>0$. We may therefore assume that $G$ is amenable with finitely additive left-invariant mean $\nu$ such that $\pc_\nu(G)>0$, say $\pc_\nu(G)=\alpha$. Let $r\in\N$ large enough to apply \cref{indep.cosets} (B) with the mean $\nu$, and define $\Gamma_r$ to be the group generated by $\{x\in G:[G:C_G(x)]\le r\}$ as in that proposition.

Let $\mu$ be an arbitrary finitely additive left-invariant mean on $G$, and let $M=(\mu_n)_{n=1}^\infty$ be an arbitrary sequence of probability measures on $G$ that measures index uniformly, and note that
\begin{equation}\label{eq:pc.split}
\pc_\mu(G)=\int_{x\in\Gamma_r}\mu(C_G(x))\diff\mu(x)+\int_{x\in G\backslash\Gamma_r}\mu(C_G(x))\diff\mu(x)
\end{equation}
and
\begin{equation}\label{eq:dc.split}
\pc_{\mu_n}(G)=\sum_{x\in\Gamma_r}\mu_n(C_G(x))\mu_n(x)+\sum_{x\in G\backslash\Gamma_r}\mu_n(C_G(x))\mu_n(x).
\end{equation}
The definition of $\Gamma_r$ implies that the centraliser of any given $x\in G\backslash\Gamma_r$ has index at least $r$, which implies that
\begin{equation}\label{eq:pc.rubbish}
\int_{x\in G\backslash\Gamma_r}\mu(C_G(x))\diff\mu(x)\le\frac{1}{r},
\end{equation}
and also, by definition of uniform measurement of index, that
\begin{equation}\label{eq:dc.rubbish}
\limsup_{n\to\infty}\sum_{x\in G\backslash\Gamma_r}\mu_n(C_G(x))\mu_n(x)\le\frac{1}{r}.
\end{equation}

\cref{indep.cosets} (B) implies that the centraliser of any given $x\in\Gamma_r$ has index at most $r^{6/\alpha+2}$. Since $G$ is finitely generated, \cref{fin.gen.fin.ind} therefore implies that we may enumerate as $C_1,\ldots,C_k$ the finite-index subgroups of $G$ arising as centralisers of elements of $\Gamma_r$. We may then partition $\Gamma_r$ into sets $A_1,\ldots,A_k$ defined by
\[
A_i=\{x\in\Gamma_r:C_G(x)=C_i\},
\]
giving
\begin{equation}\label{eq:int.indep.subgroups}
\int_{x\in\Gamma_r}\mu(C_G(x))\diff\mu(x)=\sum_{i=1}^k\mu(A_i)\mu(C_i)
\end{equation}
and
\begin{equation}\label{eq:Fol.indep.subgroups}
\sum_{x\in\Gamma_r}\mu_n(C_G(x))\mu_n(x)=\sum_{i=1}^k\mu_n(A_i)\mu_n(C_i)
\end{equation}
for every $n$.

However, for each $i$ we have
\[
A_i=C_{\Gamma_r}(C_i)\left\backslash\bigcup_{H\gneqq C_i}C_{\Gamma_r}(H)\right..
\]
Writing $Z$ for the set of subgroups $H\gneqq C_i$ such that $C_{\Gamma_r}(H)$ has infinite index in $G$ and $P$ for the set of subgroups $H\gneqq C_i$ such that $C_{\Gamma_r}(H)$ has finite index in $G$, and noting that each of $Z$ and $P$ is finite by \cref{fin.gen.fin.ind}, we therefore have
\[
A_i=\left(C_{\Gamma_r}(C_i)\left\backslash\bigcup_{H\in P}C_{\Gamma_r}(H)\right.\right)\left\backslash\bigcup_{H\in Z}C_{\Gamma_r}(H)\right.
\]
However, since $P$ is finite the union $\bigcup_{H\in P}C_{\Gamma_r}(H)$ is also the union of finitely many cosets of the finite-index subgroup $\bigcap_{H\in P}C_{\Gamma_r}(H)$, each of which is contained entirely within $C_{\Gamma_r}(C_i)$. This implies that the right-hand side of \eqref{eq:int.indep.subgroups} takes some value $\lambda_r$ that does not depend on the choice of $\mu$, and that the right-hand side of \eqref{eq:Fol.indep.subgroups} converges to $\lambda_r$ as $n\to\infty$ for every $M$ that measures index uniformly. Combined with \eqref{eq:pc.split} and \eqref{eq:pc.rubbish}, this implies that $\lambda_r\to\pc_\mu(G)$ as $r\to\infty$, whilst combined with \eqref{eq:dc.split} and \eqref{eq:dc.rubbish} it implies that $\pc_{\mu_n}(G)\to\dc_M(G)=\dc_\mu(G)$ as $n\to\infty$.
\end{proof}

\section{Algebraic consequences of a high degree of commutativity}\label{sec:proofs}
In this section we complete the proofs of our main results. We start with \cref{main:measures} (2) and \cref{main:mean} (2), as they are ingredients in the proofs of \cref{main:measures} (1) and \cref{main:mean} (1), respectively.

\begin{proof}[Proof of \cref{main:measures} (2) and \cref{main:mean} (2)]
Given $\delta>0$ define $X_\delta=\{x\in G:[G:C_G(x)]\le2/(1+2\delta)\}$, noting that since $2/(1+2\delta)<2$ the set $X_\delta$ is central. In the setting of \cref{main:mean} (2), \cref{large.centraliser} implies that for $\delta<\eps$ we have $\mu(X_\delta)\ge\eps-\delta$, and so letting $\delta\to0$ we see that $\mu(Z(G))\ge\eps$, and \cref{main:mean} (2) follows. In the setting of \cref{main:measures} (2), \cref{measures=>detects} and \cref{large.centraliser.detect} imply that for $\delta<\eps$ we have $\limsup_{n\to\infty}\mu_n(X_\delta)\ge\eps-\delta$, and so letting $\delta\to0$ we see that $\limsup_{n\to\infty}\mu_n(Z(G))\ge\eps$. \cref{main:measures} (2) then follows from the fact that the $\mu_n$ measure index uniformly.
\end{proof}

The conclusion of \cref{main:measures} (2) and \cref{main:mean} (2) allows us to give an explicit expression for the degree of commutativity, as follows.

\begin{prop}\label{v.central.implies.indep}
Let $G$ be a countable group with a central subgroup $Z$ of finite index, and let $T$ be a left transversal for $Z$ in $G$. Then if $\mu$ is a finitely additive left-invariant mean on $G$ then
\[
\pc_\mu(G)=\frac{1}{|G/Z|}\sum_{t\in T}\frac{1}{[G:C_G(t)]}.
\]
Moreover, if $M=(\mu_n)_{n=1}^\infty$ is a sequence of probability measures on $G$ that measure index uniformly then
\[
\dc_M(G)=\frac{1}{|G/Z|}\sum_{t\in T}\frac{1}{[G:C_G(t)]},
\]
and the $\limsup$ in the definition of $\dc_M(G)$ is actually a limit.
\end{prop}
\begin{proof}
Elements $x,y\in G$ commute if and only if $xz$ and $y$ commute for every $z\in Z$. This implies that
\begin{align*}
\pc_\mu(G)&=\sum_{t\in T}\mu(tZ)\mu(C_G(t))\\
     &=\frac{1}{|G/Z|}\sum_{t\in T}\frac{1}{[G:C_G(t)]},
\end{align*}
as required.
It also implies that
\begin{align*}
\dc_{\mu_n}(G)&=\sum_{t\in T}\mu_n(tZ)\mu_n(C_G(t))\\
     &\to\frac{1}{|G/Z|}\sum_{t\in T}\frac{1}{[G:C_G(t)]} &\text{(by uniform measurement of index)},
\end{align*}
as required.
\end{proof}

\begin{proof}[Proof of \cref{main:measures} (1) and \cref{main:mean} (1)]
In this proof we write $\dc(G)$ to mean $\dc_M(G)$ in the context of \cref{main:measures}, or $\dc_\mu(G)$ in the context of \cref{main:mean}. We may assume that $G$ is not abelian and that $\pc(G)\ge\frac{5}{8}$, and prove that $\pc(G)\le\frac{5}{8}$.

Let $T$ be a left transversal in $G$ of the centre $Z$ of $G$. It follows from \cref{main:measures} (2) or \cref{main:mean} (2), as appropriate, and from \cref{v.central.implies.indep}, that
\[
\pc(G)=\frac{1}{|G/Z|}\sum_{t\in T}\frac{1}{[G:C_G(t)]}.
\]
Since only one element of $T$ belongs to $Z$, and every other $t\in T$ has $[G:C_G(t)]\ge2$, we conclude that
\begin{align*}
\pc(G)&\le\frac{1}{|G/Z|}+\frac{|G/Z|-1}{2|G/Z|}\\
     &=\frac{1}{2}+\frac{1}{2|G/Z|}.
\end{align*}
Finally, we observe as Gustafson \cite{gustafson} did that since $G$ is not abelian, the quotient $G/Z$ is not cyclic, and hence $|G/Z|\ge4$. Substituting this into the last inequality then gives $\pc(G)\le\frac{5}{8}$, as required.
\end{proof}

\begin{proof}[Proof of \cref{main:mean} (3)]
We essentially reproduce Neumann's proof of \cref{neumann}. Let $r=\frac{1}{\alpha^2}+\frac{1}{\alpha}+1$, let $\Gamma$ be the group $\Gamma_r$ defined in \cref{indep.cosets} and set $H=[\Gamma,\Gamma]$, so that $\Gamma/H$ is abelian. Since the set $Y_r$ generating $\Gamma_r$ in \cref{indep.cosets} is invariant under conjugation by $G$, the subgroup $\Gamma$ is normal in $G$, and since $H$ is characteristic in $\Gamma$ it is normal in $G$.

Conclusion (A) of \cref{indep.cosets} implies that the index of $\Gamma$ in $G$ is less than $\alpha^{-1}+1$, and hence at most $\lceil\alpha^{-1}\rceil$, as required. On the other hand, conclusion (B) implies that the maximum size of a conjugacy class in $\Gamma$ is at most $(3\alpha^{-2})^{6/\alpha+2}$, and so \cref{bfc} implies that the cardinality of $H$ is at most $(2\alpha^{-1})^{O(\alpha^{-2}\log\alpha^{-1})}$, as required.
\end{proof}

\begin{proof}[Proof of \cref{main:measures} (3)]
Since $G$ is finitely generated, \cref{neumann.detect} and \cref{measures=>detects} imply that $G$ is virtually abelian, and in particular amenable. It follows that $G$ has a finitely additive left-invariant mean $\mu$, and then \cref{indep} implies that $\dc_\mu(G)=\alpha$. The desired result therefore follows from \cref{main:mean} (3).
\end{proof}

\begin{remark}\label{rem:mean}
In the specific case where $M$ is a sequence of almost-invariant measures, one can prove \cref{main:measures} (3) directly by a similar proof to \cref{main:mean} (3), without using \cref{indep}. However, it does not appear that such an argument is possible when $M$ is assumed more generally to be a sequence of measures that measures index uniformly. This is due to the use of  \cref{indep.cosets} in the proof of \cref{main:mean} (3), and in particular the use of \cref{neum.translates} in the proof of \cref{indep.cosets}. More precisely, \cref{neum.translates} uses the invariance of the invariant mean $\mu$, and whilst the almost-invariance of a sequence of almost-invariant measures is enough to prove a version of it, for more general sequences of measures that measure index uniformly there does not appear to be any reason to expect it to hold. 

One reason one might be interested in proving \cref{main:measures} (3) directly for almost-invariant measures is that the existence of an invariant mean on a virtually abelian group requires the axiom of choice; one could avoid this in the proof of \cref{main:measures} (3) by replacing \cref{main:mean} (3) with the special case of \cref{main:measures} (3) for almost-invariant measures. We leave it to the interested reader to prove \cref{main:measures} (3) in this way.
\end{remark}

\begin{proof}[Proof of \cref{neum.converse}]
Since $G$ is amenable, \cref{indep} implies that it is sufficient to prove that $\pc_\mu(G)\ge\frac{1}{m^2d}$ for every finitely additive left-invariant mean $\mu$.

Fix $x\in\Gamma$. The fact that $\Gamma/H$ is abelian implies that $[x,y]\subset H$ for every $y\in\Gamma$. Enumerating the elements of $H$ as $c_1,\ldots,c_d$, we may therefore partition $\Gamma$ into sets $\Gamma_1,\ldots,\Gamma_d$ defined by
\[
\Gamma_i=\{y\in\Gamma:[x,y^{-1}]=c_i\}.
\]
Applying a technique that was used in \cite{nfdl}, for each $i$ with $\Gamma_i\ne\varnothing$ we fix some arbitrary $y_i\in\Gamma_i$ and note that $y_i^{-1}y\in C_G(x)$ for every $y\in\Gamma_i$. In particular,
\begin{equation}\label{eq:partition.central}
y_i^{-1}\Gamma_i\subset C_G(x)
\end{equation}
for every such $i$. Finite additivity and the pigeonhole principle imply that there exists some $i$ with
\begin{align*}
\mu(\Gamma_i)&\ge\frac{1}{d}\mu(\Gamma)\\
    &=\frac{1}{md},
\end{align*}
and then \eqref{eq:partition.central} and left invariance imply that $\mu(C_G(x))\ge\frac{1}{md}$. Since this holds for every $x\in\Gamma$, the proposition is proved.
\end{proof}

\begin{proof}[Proof of \cref{indep.from.mu.and.F}]
Case (1) follows from \cref{main:mean} (3) and \cref{neum.converse}. Case (2) follows from \cref{main:measures} (2), \cref{main:mean} (2) and \cref{v.central.implies.indep}.
\end{proof}

\section{An application to conjugacy ratios}\label{sec:cr}
Given a finitely generated group $G$ with finite symmetric generating set $S$ containing the identity, and writing $\mathcal{C}(G)$ for the set of conjugacy classes of $G$, Cox \cite{cox} introduces the \emph{conjugacy ratio} $\crat_S(G)$ of $G$ with respect to $S$ via
\begin{equation}\label{eq:crat}
\crat_S(G)=\limsup_{n\to\infty}\frac{|\{C\in\mathcal{C}(G):C\cap S^n\ne\varnothing\}|}{|S^n|}.
\end{equation}
For notational convenience, in the present paper we also define a $\liminf$ version,
\[
\crat^\ast_S(G)=\liminf_{n\to\infty}\frac{|\{C\in\mathcal{C}(G):C\cap S^n\ne\varnothing\}|}{|S^n|}.
\]
It is well known that in a finite group the conjugacy ratio coincides with the degree of commutativity; Cox asks whether this remains true for infinite groups \cite[Question 2]{cox}, and also whether $\crat_S(G)>0$ if and only if $G$ is virtually abelian \cite[Question 1]{cox}. Positive answers to \cref{conj:amv} and the first of these questions would of course imply a postive answer to the second.

Ciobanu, Cox and Martino conjecture explicitly that $\crat_S(G)>0$ if and only if $G$ is virtually abelian  \cite[Conjecture 1.1]{laura}, and prove it for residually finite groups satisfying \eqref{eq:subexp.growth} \cite[Theorem 3.7]{laura}, as well as for hyperbolic groups, lamplighter groups and right-angled Artin groups \cite[\S4]{laura}. In this section we answer both of Cox's questions in the affirmative in the case where $S$ satisfies \eqref{eq:subexp.growth}; in particular, this removes the need for the residually finite hypothesis from the result of Ciobanu, Cox and Martino.

Before we state our results precisely we need some further notation. A sequence $(F_n)_{n=1}^\infty$ of finite subsets of $G$ is said to be a \emph{right-F\o lner sequence} if
\[
\frac{|F_nx\triangle F_n|}{|F_n|}\to0
\]
as $n\to\infty$ for every $x\in G$. We have already noted that if $S$ satisfies \eqref{eq:subexp.growth} then $(S^n)_{n=1}^\infty$ is a left-F\o lner sequence; by symmetry it is also a right-F\o lner sequence.

By analogy with \eqref{eq:crat}, we define the \emph{conjugacy ratio} $\crat_F(G)$ of $G$ with respect to a F\o lner sequence $F=(F_n)_{n=1}^\infty$ of finite subsets of $G$ via
\[
\crat_F(G)=\limsup_{n\to\infty}\frac{|\{C\in\mathcal{C}(G):C\cap F_n\ne\varnothing\}|}{|F_n|}.
\]
We also abuse notation slightly and identify $F$ with the sequence $(\mu_n)_{n=1}^\infty$ of uniform probability measures on the sets $F_n$, so that the notation $\dc_F(G)$ makes sense.

It is an easy exercise to check that in a finite group the conjugacy ratio is always equal to the degree of commutativity. Here we show that this remains true in an infinite group when both are defined with respect to a simultaneous left- and right-F\o lner sequence.
\begin{theorem}\label{cg=dc}
Let $G$ be a finitely generated amenable group, and let $F=(F_n)_{n=1}^\infty$ be both a left- and right-F\o lner sequence for $G$. Then $\crat_F(G)=\dc_F(G)$.
\end{theorem}
In particular, combined with Theorems \ref{main:measures} and \ref{folner.uniform} this implies the following.
\begin{corollary}\label{cg0.iff.dc0}
Let $G$ be a finitely generated amenable group, and let $F=(F_n)_{n=1}^\infty$ be both a left- and right-F\o lner sequence for $G$. Then the following hold.
\begin{enumerate}
\item If $\crat_F(G)>\frac{5}{8}$ then $G$ is abelian.
\item If $\crat_F(G)\ge\frac{1}{2}+\eps$ for some $\eps>0$ then the centre of $G$ has index at most $\frac{1}{\eps}$ in $G$.
\item If $\crat_F(G)\ge\alpha>0$ then $G$ has a normal subgroup $\Gamma$ of index at most $\lceil\alpha^{-1}\rceil$ and a normal subgroup $H$ of cardinality at most $(2\alpha^{-1})^{O(\alpha^{-2}\log\alpha^{-1})}$ such that $H\subset\Gamma$ and $\Gamma/H$ is abelian.
\end{enumerate}
\end{corollary}
\begin{remark*}
It is well known that every finitely generated amenable group admits sequences that are simultaneously left- and right-F\o lner, and so \cref{cg=dc} and \cref{cg0.iff.dc0} have content for every finitely generated amenable group.
\end{remark*}
If $S$ is a finite symmetric generating set for $G$ containing the identity and satisfying \eqref{eq:subexp.growth} then $(S^n)_{n=1}^\infty$ is both a left- and right-F\o lner sequence, and so Corollary \ref{cg0.iff.dc0} is a generalisation and strengthening of Ciobanu, Cox and Martino's result \cite[Theorem 3.7]{laura}. Moreover, if $G$ is a finitely generated group of subexponential growth and $S$ is a finite symmetric generating set containing the identity then it is well known and easy to check that some subsequence of $(S^n)_{n=1}^\infty$ is both a left- and right-F\o lner sequence, and so \cref{cg0.iff.dc0} also implies the following.
\begin{corollary}
Let $G$ be a finitely generated group of subexponential growth, and let $S$ be a finite symmetric generating set containing the identity. Then the following hold.
\begin{enumerate}
\item If $\crat^\ast_S(G)>\frac{5}{8}$ then $G$ is abelian.
\item If $\crat^\ast_S(G)\ge\frac{1}{2}+\eps$ for some $\eps>0$ then the centre of $G$ has index at most $\frac{1}{\eps}$ in $G$.
\item If $\crat^\ast_S(G)\ge\alpha>0$ then $G$ has a normal subgroup $\Gamma$ of index at most $\lceil\alpha^{-1}\rceil$ and a normal subgroup $H$ of cardinality at most $(2\alpha^{-1})^{O(\alpha^{-2}\log\alpha^{-1})}$ such that $H\subset\Gamma$ and $\Gamma/H$ is abelian.
\end{enumerate}
\end{corollary}

An important ingredient in the proof of \cref{cg=dc} is the following reformulation of the definition of the conjugacy ratio with respect to a left- and right-F\o lner sequence.
\begin{prop}\label{contain.most.classes}
If $F$ is both left- and right-F\o lner then 
\[
\frac{|\{C\in\mathcal{C}(G):C\cap F_n\ne\varnothing\text{ and }C\not\subset F_n\}|}{|F_n|}\to 0.
\]
In particular,
\[
\crat_F(G)=\limsup_{n\to\infty}\frac{|\{C\in\mathcal{C}(G):C\subset F_n\}|}{|F_n|}.
\]
\end{prop}
\begin{proof}
Let $S$ be a finite generating set for $G$. The fact that $F$ is both left- and right-F\o lner implies that
\[
\frac{|S^{-1}F_nS\backslash F_n|}{|F_n|}\to0,
\]
and hence in particular that
\[
\frac{|\{C\in\mathcal{C}(G):C\cap(S^{-1}F_nS\backslash F_n)\ne\varnothing\}|}{|F_n|}\to0.
\]
Note, however, that if $X$ is any subset of a conjugacy class $C$ such that $S^{-1}XS\cap C=X$ then in fact $X=C$, which implies in particular that
\[
\{C\in\mathcal{C}(G):C\cap F_n\ne\varnothing\text{ and }C\not\subset F_n\}\subset\{C\in\mathcal{C}(G):C\cap(S^{-1}F_nS\backslash F_n)\ne\varnothing\},
\]
and so the result follows.
\end{proof}

\begin{proof}[Proof of \cref{cg=dc}]
Writing
\[
\eps_n(x)=\frac{|C_G(x)\cap F_n|}{|F_n|}-\frac{1}{[G:C_G(x)]},
\]
we have
\[
\dc(F_n)=\frac{1}{|F_n|}\sum_{C\in\mathcal{C}(G):C\cap F_n\ne\varnothing}
                          \left(\sum_{x\in C\cap F_n}\frac{1}{[G:C_G(x)]}\right)+\frac{1}{|F_n|}\sum_{x\in F_n}\eps_n(x),
\]
\cref{folner.uniform} implies that $\eps_n(x)\to0$ uniformly in $x$ as $n\to\infty$, and so this implies that
\begin{equation}\label{eq:dc.with.ccs}
\dc(F_n)=\frac{1}{|F_n|}\sum_{C\in\mathcal{C}(G):C\cap F_n\ne\varnothing}\left(\sum_{x\in C\cap F_n}\frac{1}{[G:C_G(x)]}\right)+o(1).
\end{equation}
The index of $C_G(x)$ is equal to the cardinality of the conjugacy class of $x$, and so
\[
\sum_{x\in C\cap F_n}\frac{1}{[G:C_G(x)]}\le1
\]
for every $C\in\mathcal{C}(G)$, which with \eqref{eq:dc.with.ccs} implies that $\dc_F(G)\le\crat_F(G)$.

On the other hand, \eqref{eq:dc.with.ccs} also implies that
\begin{align*}
\dc_F(G)&\ge\limsup_{n\to\infty}\frac{1}{|F_n|}\sum_{C\in\mathcal{C}(G):C\subset F_n}\left(\sum_{x\in C}\frac{1}{[G:C_G(x)]}\right)\\
      &=\limsup_{n\to\infty}\frac{1}{|F_n|}\sum_{C\in\mathcal{C}(G):C\subset F_n}\left(\sum_{x\in C}\frac{1}{|C|}\right)\\
      &=\limsup_{n\to\infty}\frac{|\{C\in\mathcal{C}(G):C\subset F_n\}|}{|F_n|}\\
     &=\crat_F(G),
\end{align*}
the last equality being by \cref{contain.most.classes}.
\end{proof}

\cref{neum.converse} and \cref{cg=dc} combine to give a converse to \cref{cg0.iff.dc0} when $F$ is both a left- and right-F\o lner sequence. However, it turns out to be extremely straightforward to prove this converse directly for just left-F\o lner sequences, as follows.
\begin{prop}[converse to \cref{cg0.iff.dc0}]
Let $G$ be a group with a subgroup $\Gamma$ of index $m$ and a subgroup $H\lhd\Gamma$ of cardinality  $d$ such that $\Gamma/H$ is abelian, and let $F=(F_n)_{n=1}^\infty$ be a left-F\o lner sequence for $G$. Then $\crat_F(G)\ge\frac{1}{m^2d}$.
\end{prop}
\begin{remark*}
In particular, since a finitely generated finite-by-abelian group $G$ is virtually abelian, if $F$ is both a left- and right-F\o lner sequence then $\crat_F(G)>0$ if and only if $G$ is virtually abelian.
\end{remark*}
\begin{proof}
Let $x\in\Gamma$. Since $[\Gamma,\Gamma]\subset H$, for every $y\in\Gamma$ we have $y^{-1}xy\in xH$. Fixing a right transversal $T$ for $\Gamma$ in $G$, for every $t\in T$ and every $y\in\Gamma$ we therefore have $t^{-1}y^{-1}xyt\in t^{-1}xHt$. The conjugacy class of $x$ in $G$ therefore has size at most $|T||H|=md$. Since this is true for every $x\in H$, we conclude that
\[
|\{C\in\mathcal{C}(G):C\cap F_n\ne\varnothing\}|\ge\frac{|\Gamma\cap F_n|}{md}.
\]
Since \cref{folner.uniform} implies that
\[
\frac{|\Gamma\cap F_n|}{|F_n|}\to\frac{1}{m},
\]
the result follows.
\end{proof}

\end{document}